%% file: Engel_tori.tex
\documentclass{amsart} 

\usepackage{amsmath,amssymb,amsthm} 
\usepackage{graphicx} 
\usepackage[arrow, matrix, curve]{xy} 
\usepackage{color} 
\usepackage{fullpage}
\usepackage{hyperref} 

\DeclareMathOperator{\tb}{tb}
\DeclareMathOperator{\rot}{rot}
\DeclareMathOperator{\lk}{lk}
\DeclareMathOperator{\slf}{sl}

\newcommand{\R}{\mathbb{R}}
\newcommand{\Z}{\mathbb{Z}}

\newcommand{\N}{\mathbb{N}}

\newcommand{\xist}{\xi_{\mathrm{st}}}
\newcommand{\xiot}{\xi_{\mathrm{ot}}}
\newcommand{\Dst}{\mathcal{D}_{\mathrm{st}}}
\newcommand{\Est}{\mathcal{E}_{\mathrm{st}}}
\newcommand{\D}{\mathcal{D}}
\newcommand{\E}{\mathcal{E}}
\newcommand{\W}{\mathcal{W}}
\newcommand{\Tstab}{T_{\mathrm{stab}}}
\newcommand{\Kstab}{K_{\mathrm{stab}}}

\newtheoremstyle{thm}{}{}{\itshape}{}{\bfseries}{}{ }{} 
\newtheoremstyle{definition}{}{}{}{}{\bfseries}{}{ }{} 

\theoremstyle{thm}
\newtheorem{Theorem}{Theorem}[section]
\newtheorem{thm}[Theorem]{Theorem}
\newtheorem{lem}[Theorem]{Lemma}

\newtheorem*{Theorem-ohne}{Theorem}

\theoremstyle{definition}
\newtheorem{defi}[Theorem]{Definition}

\newtheorem{ex}[Theorem]{Example}

\begingroup\expandafter\expandafter\expandafter\endgroup
\expandafter\ifx\csname pdfsuppresswarningpagegroup\endcsname\relax
\else
\fi

\definecolor{amaranth}{rgb}{0.9, 0.17, 0.31} 
\definecolor{carrotorange}{rgb}{0.93, 0.57, 0.13} 
\definecolor{citrine}{rgb}{0.89, 0.82, 0.04} 
\definecolor{dartmouthgreen}{rgb}{0.05, 0.5, 0.06} 
\definecolor{ballblue}{rgb}{0.13, 0.67, 0.8} 
\definecolor{ceruleanblue}{rgb}{0.16, 0.32, 0.75} 
\definecolor{amethyst}{rgb}{0.6, 0.4, 0.8} 
\definecolor{amber}{rgb}{1.0, 0.75, 0.0} 
\definecolor{burlywood}{rgb}{0.87, 0.72, 0.53} 

\begin{document}


\title[Non-isotopic transverse tori in Engel manifolds]{Non-isotopic transverse tori in Engel manifolds} 

\author{Marc Kegel}
\address{Humboldt-Universit\"at zu Berlin, Rudower Chaussee 25, 12489 Berlin, Germany.}
\email{kegemarc@math.hu-berlin.de, kegelmarc87@gmail.com}


\begin{abstract}
In every Engel manifold we construct an infinite family of pairwise non-isotopic transverse tori that are all smoothly isotopic. To distinguish the transverse tori in the family we introduce a homological invariant of transverse tori that is similar to the self-linking number for transverse knots in contact $3$-manifolds. Analogous results are presented for Legendrian tori in even contact $4$-manifolds. 
\end{abstract}

\date{\today} 



\maketitle


\section{Introduction}

In recent years, some spectacular breakthroughs in the theory of Engel structures were obtained that have renewed interest in the field. After the solution of the existence question for Engel structures on parallelizable 4-manifolds by Vogel~\cite{Vo09}, a number of flexibility results, formulated in the language of $h$-principles, were proven~\cite{CPPP17,PV20,CPP20,PP19}. (For more details consult Section~\ref{section:back} and the references therein.)

In~\cite{PV20} del Pino and Vogel have introduced the notion of an overtwisted Engel structure, fulfilling an h-principle. $2$-tori transverse to an Engel structure play a crucial role in this construction, similar as transverse knots played a key role in the construction of overtwisted contact structures by Martinet~\cite{Ma71,Ge08}.
This suggests that transverse tori in Engel manifolds will play a similar prominent role as transverse knots in contact $3$-manifolds as started with Bennequin's work~\cite{Be82}. This would fit with the observation that in many cases a lot of the underlying topology and geometry of manifolds is encoded in its codimension-$2$ knot theory. The aim of this article is to show that the knot theory of transverse tori is rich in examples that are smoothly isotopic but not isotopic as transverse tori. 


\begin{thm}\label{thm:mainEng}
In any Engel manifold $(M,\D)$ there exist infinitely many pairwise non-isotopic transverse $2$-tori that are all smoothly isotopic.
\end{thm}

This result was independently obtained by Gompf in~\cite{Go22}.

To prove Theorem~\ref{thm:mainEng} we will define in Section~\ref{section:transverseTori} a stabilization operation for transverse tori coming from the del Pino--Vogel construction~\cite{PV20}, that does not change the smooth isotopy class. In order to distinguish a transverse torus from its stabilization, we construct in Section~\ref{section:linkingClass} and~\ref{section:selfLinkingClass} a homological invariant, depending only on the formal data of a transverse torus, that we will use to distinguish stabilized transverse tori. This invariant can be thought of as an analogue of the self-linking number for transverse knots in contact $3$-manifolds. 


We have a full $h$-principle for transverse immersions of surfaces~\cite{PP19}. On the other hand, it remains open if transverse tori fulfill a full $h$-principle. Although, in recent work of Gompf~\cite{Go22} it is shown that transverse tori fulfill an existence $h$-principle: Every $2$-torus with trivial normal bundle in an Engel manifold is isotopic to a transverse torus.
For $1$-dimensional submanifolds of Engel manifolds the situation is solved. 
The subspace of tangent knots (also called horizontal or Engel knot) that are not everywhere tangent to $\W$ fulfill a full h-principle~\cite{CP18}, cf.~\cite{Ad07,Ge08a,FMP20}. Similarly, transverse curves (i.e.\ knots always transverse to $\D$) fullfill a full h-principle~\cite[Theorem~4.6.2]{EM02}.

Any Engel structure $\D$ induces naturally an even contact structure by $\E=[\D,\D]$ (see Section~\ref{section:back} for details). Motivated by Theorem~\ref{thm:mainEng} we want to study codimension-$2$ embeddings in even contact manifolds that fit to the even contact structure. Here the natural objects are \textit{Legendrian tori}, embedded $2$-tori that are everywhere tangent to the even contact structure $\E$. With similar ideas as in the proof of Theorem~\ref{thm:mainEng} we will prove an analogous statement for Legendrian tori in even contact structures in Section~\ref{section:EvenLegendrian}.

\begin{thm}\label{thm:mainEven}
	There exist even contact $4$-mani\-folds containing infinite families of pairwise non-isotopic Legendrian $2$-tori that are all smoothly isotopic.
\end{thm}


First, we will give some relevant background on Engel and even contact structures in Section~\ref{section:back}. In Section~\ref{section:transverseTori} we will recall some facts about transverse tori from~\cite{PV20}. In particular, we will review a method of constructing transverse tori in Engel manifolds from transverse knots in contact $3$-manifolds and define a stabilization operation that will not change the smooth isotopy class of the transverse torus. To show that this stabilization operation changes the isotopy class as a transverse torus we devote Section~\ref{section:linkingClass} to \textit{the linking class}, a homology class, which can be seen as a generalization of the linking number, defined for general subsets in topological spaces.
From this we can define the self-linking class in Section~\ref{section:selfLinkingClass} and deduce a way of computing it for transverse tori coming from transverse knots via the del Pino--Vogel construction of transverse tori. Using this method we will distinguish transvere tori that are all smoothly isotopic.
The remaining section is devoted to Legendrian tori in even contact structures and to prove Theorem~\ref{thm:mainEven}.

\subsection*{Conventions:} We work in the smooth and oriented category, i.e. all manifolds, maps, and ancillary objects are assumed to be smooth. In addition, all manifolds and geometric structures on them like Engel and contact structures are assumed to be oriented and cooriented. Homology groups are always understood over the integers.

\subsection*{Acknowledgments:} I am happy to thank Sebastian Durst, \'Alvaro del Pino, Robert Gompf, and Dingyu Yang for useful comments, interesting discussions, and their interest in this work. Special thanks go to Robert Gompf for detailed feedback on an earlier version of the manuscript and for correcting some errors in my earlier arguments. I also thank the referees for useful suggestions and the careful reading of this article.

This project was finished while I was supported by ICERM (the Institute for Computational and Experimental Research in Mathematics in
Providence, RI) during the semester program on Braids (Feb 1 - May 6, 2022).

\section{Background}
\label{section:back}

\subsection{Topologically stable distributions}
An \textit{Engel structure} is a maximally non-integrable tangential $2$-plane field $\D$ on a $4$-manifold $M$, i.e.\ its iterated Lie-brackets $[\D,[\D,\D]]$ give the full tangent bundle $TM$. It is well-known that Engel structures admit a local normal form, which means Engel structures can only be studied globally. 
Moreover, the set of Engel structures on a $4$-manifold $M$ is open in the set of $2$-dimensional distributions on $M$. In general, a class of $k$-dimensional distributions  on
an $n$-manifold $M$ is called \textit{topologically stable} if it is open and admits a local normal form. The classification of such distributions in terms of geometric structures on manifolds can be traced back to Cartan~\cite{Mo93,Pr16}: Every topologically stable distribution is  
\begin{itemize}
	\item a line field ($k=1$ and $n$ arbitrary),
	\item a \textit{contact structure} ($k=n-1$, $n=2m+1$ odd; locally written as the kernel of a $1$-form $\alpha$ with $\alpha\wedge (d\alpha)^m$ non-vanishing),
	\item an \textit{even contact structure} ($k=n-1$, $n=2m+2$ even; locally written as the kernel of a $1$-form $\alpha$ with $\alpha\wedge (d\alpha)^m$ non-vanishing) or
	\item an Engel structure ($k=2$ and $n=4$).
\end{itemize} 
Although Engel structures play a distinguished role in the above classification since they are the only class that exists only in the special dimension $4$, they are the least understood class of topologically stable distributions.
In dimension $4$ all these geometries are closely related. An Engel structure $\D$ on a $4$-manifold $M$ induces an even contact structure $\E$ by $\E=[\D,\D]$ and the so-called \textit{characteristic line field} $\W\subset \D$ defined by $[\W,\E]\subset \E$. Finally, every $3$-dimensional hypersurface $N$ in $M$ transverse to $\W$ carries a canonical contact structure defined by $\xi=\E\cap TN$.

It is therefore not surprising that Engel structures, contact structures and even contact structures often behave in the same way structurally. In the following short summary we will outline a few points. First of all, by definition all these structures look locally like a standard model:
$$\big(\R^3,\xist=\ker(dz-y\,dx)=\langle\partial_y,\partial_x+y\partial_z\rangle\big)$$ 
in the contact case, 
$$\big(\R^4,\Dst=\ker(dz-y\,dx)\cap\ker(dy-w\,dx)=\langle\partial_w,\partial_x+y\partial_z+w\partial_y\rangle\big)$$ 
in the Engel case and 
$$\big(\R^4,\Est=[\Dst,\Dst]=\ker(dz-y\,dx)=\langle\partial_w,\partial_y,\partial_x+y\partial_z\rangle\big)$$ 
in the case of an even contact structure.

The existence question for these structures has been solved and can be simplified by the statement that such a structure exists whenever there is no obvious obstruction coming from the underlying algebraic topology. Some relevant references are~\cite{Gr86, EM02, Ma71, El89, Ge08, Mc87, Vo09, CPPP17,PV20, BEM15}. On the other hand, the uniqueness question up to homotopy through such structures is in general open. 

\subsection{Contact structures}
We first discuss the $3$-dimensional contact case.  We have to distinguish the cases where the underlying $3$-manifold is open or closed. 
From Gromov's work~\cite{Gr86}, cf.~\cite{EM02}, it follows that there is an $h$-principle for contact structures on open manifolds. Roughly speaking an $h$-\textit{principle} holds if the inclusion of a genuine structure into the formal data of the structure is a weak homotopy equivalence. In this situation it implies that two positive contact structures on an open manifold are homotopic through positive contact structures if and only if they are homotopic through tangential $2$-plane fields. Thus, the problem of classifying contact structures on open manifolds up to homotopy reduces completely to algebraic topology methods. 

However, for open manifolds there is in general a difference between homotopy through contact structures and isotopy. The first such example was found by Bennequin~\cite{Be82}. He showed that there exists a so-called overtwisted contact structure $\xiot$ on $\R^3$ that is homotopic through contact structures but non-isotopic to the standard structure $\xist$. This can be proven by studying \textit{Legendrian knots}, always tangent to the contact structures, or \textit{transverse knots}, always transverse to the contact structures, and their \textit{classical invariants}: the \textit{self-linking number} $\slf$ for transverse knots and the \textit{Thurston--Bennequin invariant} $\tb$ and the \textit{rotation number} $\rot$ for Legendrian knots. 

For closed manifolds the situation is different. Eliashberg~\cite{El89} distinguishes two classes of contact structures: the \textit{overtwisted} contact structures (containing a Legendrian unknot with vanishing Thurston--Bennequin invariant) and the \textit{tight} contact structures (containing no such unknot). Then he shows that the overtwisted contact structures fulfill an $h$-principle, implying that two positive overtwisted contact structures are homotopic through positive contact structures if and only if they are homotopic as tangential $2$-plane fields. On closed manifolds we have Gray's theorem, implying that two contact structures that are homtopic through contact structures are also isotopic~\cite{Ge08}.
 
On the other hand, the tight contact structures are closely linked to the underlying topology of the manifold and do not only depend on the algebraic topology of the underlying $2$-plane field. As in the case of $\R^3$, tight contact structures are closely related to their Legendrian and transverse knot theory. In fact, a contact manifold is tight if and only if every nullhomologous transverse knot $T$ fulfills the Bennequin bound $\slf(T)\leq2g(T)-1$ or equivalently if every nullhomologous Legendrian knot $L$ fulfills $\tb(L)+\rot(L)\leq2g(L)-1$, where $g$ denotes the genus of the underlying smooth knot type~\cite{El93}.
 
Recent development shows that the situation in higher dimensions is formally the same~\cite{BEM15}. 
\subsection{Even contact structures}
For even contact structures and Engel structures we can ask how much of the above discussion carries over. First, we discuss the case of even contact structures on $4$-manifolds. For open manifolds there is again an $h$-principle by the work of Gromov~\cite{Gr86}. Surprisingly, we also have a full $h$-principle for even contact structures on closed manifolds by work of McDuff~\cite{Mc87}, so we do not need to restrict to a subclass. However, Gray stability does not hold in general, not even for closed manifolds~\cite{Mo99}. The reason is that every even contact structure $\E$ comes together with a so-called \textit{characteristic line field} $\W\subset\E$ uniquely defined by $[\W,\E]\subset \E$ and a small perturbation of $\E$ can change the dynamics of $\W$ completely. (A version of Gray's theorem holds if we can assume the characteristic line field to stay fixed during the homotopy between the even contact structures~\cite{Go97}.)

In conclusion, we can say that even contact structures on open or closed manifolds will behave structurally similar to contact structures on open manifolds, i.e.\ the study of even contact structures up to homotopy reduces completely to algebraic topology but even contact structures up to isotopy are more interesting to study.

\subsection{Engel structures}
The situation for Engel structures is less well understood. First, we observe that any Engel structure $D$ on a $4$-manifold $M$ induces a flag of its tangent bundle
$$\W\subset\D\subset\E\subset TM,$$
where $\E$ is the associated even contact structure and $\W$ its characteristic line field. Thus, in the oriented case, a $4$-manifold admitting an Engel structure has to be parallelizable. On the other hand, this is the only obstruction to the existence of an Engel structure: Every parallelizable $4$-manifold admits an Engel structure. This follows for open manifolds again by Gromov's $h$-principle and was proven for closed manifolds by Vogel~\cite{Vo09}, cf.~\cite{CPPP17}.

If a full $h$-principle holds is unknown, but recently the subclass of \textit{overtwisted} Engel structures fulfilling a full $h$-principle was introduced~\cite{PV20}. One of the main questions in the field of Engel structures asks if there are Engel structures that are not overtwisted. Moreover, it is not clear if overtwistedness is a property that is preserved under homotopy of Engel structures. And thus we can more generally ask if there exist Engel structures that are not homotopic to overtwisted Engel structures. For Engel structures there also exists the class of \textit{loose} Engel structures that satisfy an h-principle, as well~\cite{CPP20}. Unlike the overtwisted Engel structures, that are characterized by a local property, these loose Engel structures are given by a global property. It is not clear how loose Engel structures are related to overtwisted ones.

Since Gray's theorem for Engel structures does not hold~\cite{Mo99}, there can be also Engel structures that are homotopic but not isotopic. The first such examples were found on $\R^4$ by Gershkovich~\cite{Ge95}. One possibility to construct such an example is by using an overtwisted contact structure $\xiot$ on $\R^3$ as follows. Let $X_1$ and $X_2$ be two linearly independent vector fields in $\xiot$. Then it is easy to see that $Y_1:=\partial_t$ and $Y_2:=\cos(t)X_1+\sin(t)X_2$ generate an Engel structure on $\R^4$, where $t$ is a coordinate on the extra $\R$ factor. In~\cite{Ge95} Gershkovich used contact geometry to distinguish this Engel structure from the standard Engel structure on $\R^4$. Such Engel structures could potentially also be distinguished by the classical invariants of transverse tori.


\section{Transverse tori}
\label{section:transverseTori} 
In~\cite{PV20} del Pino and Vogel used embedded transverse tori in Engel manifolds as generalizations of transverse knots in contact manifolds. In this section we recall the relevant definitions and constructions from~\cite{PV20}.

\begin{defi}
	An embedded $2$-torus $T$ in an Engel manifold $(M,\D)$ is called a \textbf{transverse torus}, if it is everywhere transverse to $\D$.
\end{defi}

In principle, this definition would make sense for any embedded surface, but in the oriented case any such surface would inherit a parallelization and a trivialization of its normal bundle (both coming from the parallelization of the Engel structure). Thus, if we restrict to oriented closed surfaces, every transverse surface is a torus. We remark that the same reasoning applies for non-orientable closed surfaces and we observe that any non-orientable closed transverse surface is a Klein bottle.

\begin{ex}\label{ex:homologous}
We consider the $4$-torus $T^4$ with angular coordinates $(\theta_1,\theta_2,\theta_3,\theta_4)$ and Engel structure given by the span of $\partial_{\theta_1}$ and $\cos\theta_1\partial_{\theta_2}+\sin\theta_1\partial_{\theta_3}+\partial_{\theta_4}$. Then the $2$-torus given by $\{\theta_1=\theta_2=0\}$ is a transverse torus.
\end{ex}

In~\cite{PV20} the following construction of transverse tori in standard neighborhoods of transverse curves is described. 
A \textbf{transverse curve} $C$ is an embedded $S^1$ in an Engel manifold $(M,\D)$ that is everywhere transverse to the induced even contact structure $\E=[\D,\D]$. By standard methods one can prove the following two lemmas, see~\cite[Proposition~2.7 and Lemma~2.8]{PV20} for details.

\begin{lem}\label{lem:h}
Let $C\colon S^1\rightarrow (M,\D)$ be an embedding. Then $C$ is isotopic to a transverse curve. Moreover, if $C_1$ and $C_2$ are two transverse curves in $(M,\D)$ that are smoothly isotopic, then they are isotopic as transverse curves. 
\end{lem}

\begin{lem}\label{lem:standard}
Any transverse curve $C$ in $(M,\D)$ admits a neighborhood $(\nu C,\D)$ (where here $\D$ denotes the restriction of the Engel structure of $M$ to $\nu C$) which is, for some $R>0$, Engel diffeomorphic to the following standard neighborhood
$$\big(S^1\times D_R^3,\Dst=\ker(d\theta-y\,dx)\cap\ker(dy-w\,dx)\big)$$
where $\theta$ is an angular coordinate of $S^1$, $(x,y,w)$ are Cartesian coordinates on the disk $D_R^3$ with radius $R$ in $\R^3$ and the Engel-diffeomorphism sends $C$ to $S^1\times\{0\}$.
\end{lem}

In particular there exist many transverse curves and in any standard neighborhood we can construct transverse tori with the following lemma, see~\cite[Lemma~3.3]{PV20} for details.

\begin{lem}\label{lem:delPinoVogel}
Let $K$ be a transverse knot in $(D_R^3,\xi=\ker dy-w\,dx)$. Then 
$$S^1\times K\subset(S^1\times D_R^3,\Dst)$$
is a transverse torus.
\end{lem}

With Lemmas~\ref{lem:h},~\ref{lem:standard} and~\ref{lem:delPinoVogel} we get transverse tori in any Engel manifold. We call a transverse torus arising like this a \textbf{del Pino--Vogel torus}. On the other hand, not every transverse torus is a del Pino--Vogel torus.  Indeed, if a transverse torus arises as above, it is nullhomologous, since it bounds the embedded $3$-manifold $S^1\times F$ where $F$ is a Seifert surface of $K$ in $D_R^3$. In particular, we see that the transverse torus from Example~\ref{ex:homologous} is not a del Pino--Vogel torus.
Moreover, it is not clear if the \textbf{core} $C$ and the \textbf{profile} $K$ in the above construction are unique.


There are different types of natural equivalence relations on transverse tori. We briefly discuss them to fix our notation.
Two transverse tori $T_1$ and $T_2$ in an Engel manifold are called \textbf{(ambient) isotopic} if there exists an Engel diffeomorphism $f\colon(M,\D)\rightarrow(M,\D)$ that maps $T_1$ to $T_2$ and is isotopic (through Engel diffeomorphisms) to the identity.
Two transverse tori $T_1$ and $T_2$ in an Engel manifold are called \textbf{equivalent} if there exists an Engel diffeomorphism $f\colon(M,\D)\rightarrow(M,\D)$ that maps $T_1$ to $T_2$.
Obviously, ambient isotopy implies equivalence, the reverse implication is in general wrong (since not every Engel diffeomorphism will be isotopic to the identity).
Two transverse tori $T_i\colon T^2\rightarrow (M,\D)$ are called \textbf{isotopic} if there exists a homotopy between them through transverse tori.
Ambient isotopy implies isotopy, but other than in the smooth category the isotopy extension theorem does not hold, i.e. there exist transverse tori that are isotopic but not ambient isotopic. Note that $\E$ intersects a transverse torus in a line field, the so called \textbf{characteristic foliation}, which will be preserved under ambient isotopy but not under isotopy. 
Two transverse tori $T_i\colon T^2\rightarrow (M,\D)$ are called \textbf{formal isotopic} if there exists a homotopy between them through formal transverse tori. Every isotopy can be also seen as a formal isotopy. In Section~\ref{sec:formal} we will properly define formal transverse tori and demonstrate that our invariants will distinguish transverse tori only up to formal isotopy. It remains open if transverse tori fulfill a full $h$-principle, i.e.\ if any two formal isotopic transverse tori are isotopic.

In Theorem~\ref{thm:mainEng} we show that in any Engel manifold there are infinitely many non-formally isotopic transverse tori that are all smoothly isotopic. The construction of these examples is via the obvious stabilization operation. Let $T$ be a del Pino--Vogel transverse torus with core $C$ and profile $K$. We define the \textbf{stabilization} $\Tstab$ of $T$ to be $C\times \Kstab$ where $\Kstab$ is the stabilization of $K$.
Since it is not clear if the profile is unique, the stabilization could depend on $K$ and $C$ and it is not clear if the transverse knot $K$ is an invariant of $T$. However, we would expect that $\Tstab$ is not isotopic to $T$. In the next sections we will develop an invariant, similar to the self-linking number of transverse knots, that will enable us to distinguish stabilizations of transverse tori.


\section{The linking class for embedded tori}
\label{section:linkingClass} 

In this section, we will introduce a generalization of the linking number to higher dimensional embedded submanifolds. We believe that this invariant, which we call linking class, is known to the experts, but we could not find a discussion in the literature. Other definitions of linking numbers in higher dimensions are mostly given for embeddings of spheres and differ from our discussion here, see for example~\cite{CKS04}.

Let $L$ and $T$ be submanifolds of dimension $l$ and $t$ of a topological $m$-manifold $M$. Then we define the $l$-th \textbf{linking class} $\lk(L,T)$ to be the homology class
$$\lk(L,T):=[L]\in H_l(M\setminus T).$$
We remark that for nullhomologous knots $L$ and $T$ in a $3$-manifold $M$ we can identify $H_1(M\setminus T)$ canonically with $H_1(M)\oplus \Z$, where the first linking class $\lk(L,T)$ lies in the $\Z$-summand and can thus be identified with an integer, the linking number.

%
%

To make a useful invariant out of the linking class, we first need to compute the homology of $M\setminus T$. With a view to our application and for concreteness we consider two special cases in the following lemmas.

\begin{lem} \label{lem:hom}
Let $T$ and $T'$ be embedded $2$-tori with trivial normal bundle in a $4$-manifold $M$. If $T$ and $T'$ are nullhomologous in $H_2(M)$ and if $H_3(M)=0$, then 
the linking class $\lk(T',T)$ takes values in a free abelian group of rank $2$, generated by the meridians $\alpha$ and $\beta$ of $T$.
\end{lem}

\begin{proof}
We choose an identification of $\nu T$ with $S^1 \times S^1\times D^2$ such that $S^1\times S^1\times \text{pt}$ is nullhomologous in $H_2(M\setminus\mathring{\nu T})$. Then we consider the following part of the Mayer-Vietoris sequence and the images of the generators under the homomorphisms.
\[\begin{array}{ccccccc}
0&{{\longrightarrow}}& \overset{\cong \Z^3}{\overbrace{H_2(\partial\nu T)}} &\overset{f}{\longrightarrow} &\overset{\cong \Z}{\overbrace{H_2(\nu T)}}\oplus H_2(M\setminus\mathring{\nu T})&\longrightarrow& H_2(M).\\
&&S^1\times S^1\times \text{pt}&\longmapsto& (S^1\times S^1\times0,0)&\longmapsto&0\\
&&\alpha:=S^1\times\text{pt}\times \partial D^2&\longmapsto& (0,\alpha)&\longmapsto&0\\
&&\beta:=\text{pt}\times S^1\times \partial D^2&\longmapsto&(0,\beta)&\longmapsto&0
\end{array}\]

The map $f$ is injective and we write for simplicity again $\alpha$ and $\beta$ for their images in $H_2(M\setminus\mathring{\nu T})$ (see Figure~\ref{fig:transversetori}). Since $T=S^1\times S^1\times 0$ is nullhomologous it maps to $0$ in $H_2(M)$. The generators $\alpha$ and $\beta$ are boundaries in $\nu T$ and thus also vanish in $H_2(M)$. Therefore, $f$ acts on the generators as stated. Since$f$ is injective, it follows that $\alpha$ and $\beta$ are primitive elements in $H_2(M\setminus\mathring{\nu T})$ and thus generate a free abelian group of rank $2$.
Since $T'$ is nullhomologous in $H_2(M)$ it follows that $T'$ seen as a homology class in $H_2(M\setminus\mathring{\nu T})$ is a linear combination of $\alpha$ and $\beta$, which implies the statement.
\end{proof}





\begin{lem} \label{lem:hom2}
Let $K$ and $K'$ be nullhomologous knots in a closed $3$-manifold $N$. We consider $L=S^1\times K$ and $L'=S^1\times K'$ in $M=S^1\times N$. Then $H_2(M\setminus L)$ is isomorphic to $H_1(N)\oplus H_2(N)\oplus \Z$ and the linking class $\lk(L',L)$ takes values in the last free abelian group generated by $S^1\times \mu_K$ where $\mu_K$ is the meridian of $K$.
\end{lem}

\begin{proof}
First we combine Poincar\'e duality, excision and the universal coefficient theorem to deduce that
\begin{equation*}
H_2(M\setminus \mathring{\nu L})\cong F_2(M,L)\oplus T_1(M,L),
\end{equation*}
where $F_2(M,L)$ denotes the free part of $H_2(M,L)$ and $T_1(M,L)$ the torsion of $H_1(M,L)$. By applying the relative K\"unneth formula to $S^1\times (N,K)$ we compute
\begin{equation*}
H_k(M,L)\cong H_{k-1}(N,K)\oplus H_k(N,K).
\end{equation*}
Next, we consider the long exact sequence of the pair $(N,K)$
\begin{equation*}
0\longrightarrow H_2(N)\longrightarrow H_2(N,K)\longrightarrow \Z\overset{0}{\longrightarrow} H_1(N)\longrightarrow H_1(N,K)\longrightarrow 0,
\end{equation*}
where the map $\Z\rightarrow H_1(N)$ is the $0$-map since $K$ is nullhomologous in $N$. Thus 
\begin{equation*}
H_1(N,K)=H_1(N)\,\text{ and }\, H_2(N,K)\cong H_2(N)\oplus\Z.
\end{equation*}
Putting everything together and using that $H_2(N)$ is a free group we obtain that $H_2(M\setminus L)$ is isomorphic to $H_1(N)\oplus H_2(N)\oplus \Z$, where we can follow the above isomorphisms to see that the $\Z$-summand is generated by $S^1\times \mu_K$.
Since $L'$ is nullhomologous in $H_2(M)$ and $K'$ is nullhomologous in $H_1(N)$ it follows that $L'$ seen as a homology class in $H_2(M\setminus L)$ is a multiple of $S^1\times \mu_K$.
\end{proof}




\section{The self-linking class for transverse tori}
\label{section:selfLinkingClass} 

Let $T$ be a transverse $2$-torus embedded in an Engel $4$-manifold $(M,\D)$. We denote by $T'$ the embedded $2$-torus that is obtained by pushing $T$ into the $\W$-direction. 
The smoothly embedded link of $2$-tori $T\cup T'$ in $M$ is an invariant of the transverse torus $T$. Similar as for the definition of the classical invariants for Legendrian or transverse knots in contact manifolds we define the following invariant for transverse tori.

\begin{defi}
Let $T$ be a transverse $2$-torus in an Engel $4$-manifold $(M,\D)$. Then we define the self-linking class $\slf(T)$ as the homology class
$$\slf(T):=\lk(T',T)\in H_2(M\setminus T).$$
\end{defi}


Next, we will see that for a del Pino--Vogel torus the self-linking class is related to the self-linking number of its profile.

\begin{lem}\label{lem:computationofsl}
Let $T$ be a del Pino--Vogel transverse torus in $(M,\D)$ with a fixed choice of a core $C$ and a profile $K$. 
In the notation introduced in the proof of Lemma~\ref{lem:hom} we can express the self-linking class of $T$ as 
\begin{equation*}
\slf(T)=\slf(K)\cdot\alpha\in H_2(M\setminus T).
\end{equation*}
\end{lem}

\begin{proof}
We get an identification of $T$ with
\begin{equation*}
S^1\times K\subset (S^1\times D^3_R,\Dst)\subset (M,\D).
\end{equation*}
In this local model $T'$ is defined to be the push-off of $T$ in the $w$-direction. In this situation $T'$ is $S^1\times K'$ where $K'$ is the push-off of $K$ in $w$-direction. If we draw the transverse knot $K\subset (D^3_R,\ker dy-wdx)$ in its front projection to the $wx$-plane, we see that $K'$ is isotopic to the blackboard framing of this projection (see Figure~\ref{fig:transversetori}).

\begin{figure}[htbp] 
	\centering
	\def\svgwidth{0,99\columnwidth}
	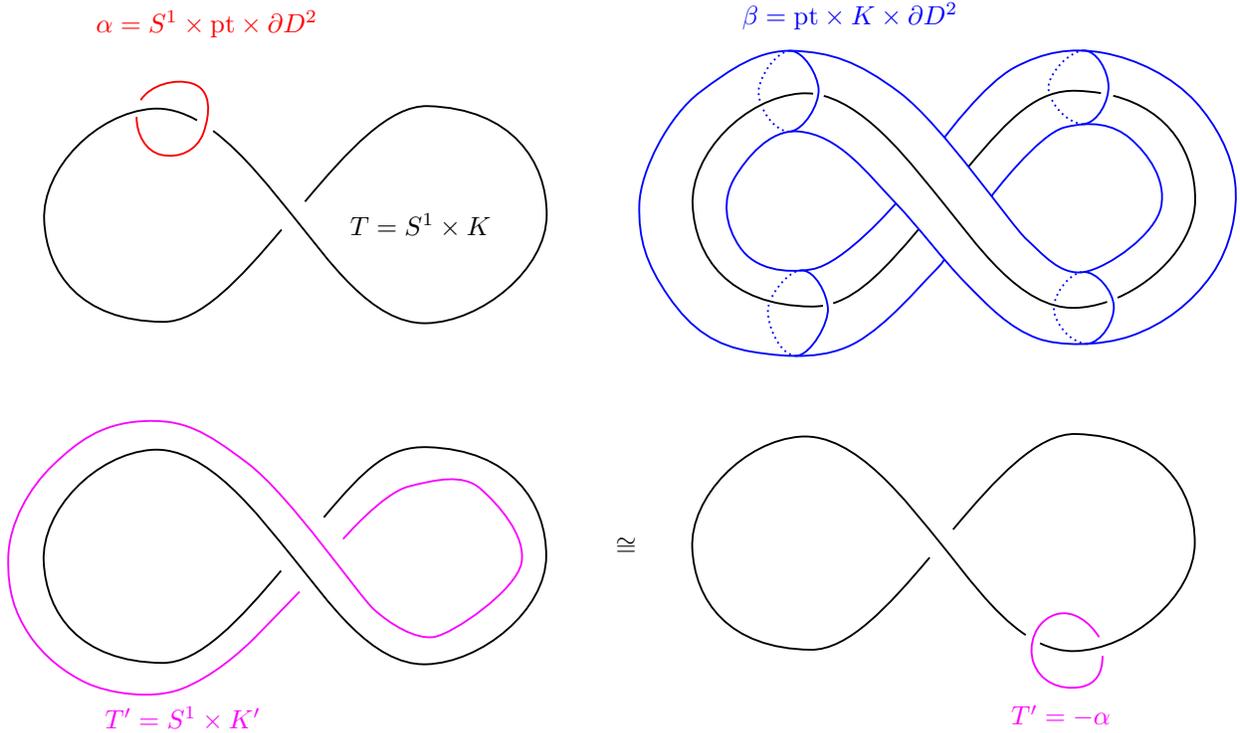
	\caption{Top row: The generators $\alpha$ and $\beta$ are presented in the front projection of an $\R^3$-slice. Bottom row: The push-off $T'$ is isotopic to $\slf(K)\alpha$.}
	\label{fig:transversetori}
\end{figure}
By an easy homology computation in $H_2(M\setminus\mathring{\nu T})$ and the definition of the self-linking number for transverse knots it follows
\begin{align*}
\slf(T)=\lk(T,T')=[T']=[S^1\times K']=\slf(K)\cdot[S^1\times\text{pt}\times\partial D^2]=\slf(K)\cdot\alpha.
\end{align*}
\end{proof}

With this preparation we are ready to give the proof of Theorem~\ref{thm:mainEng}.

\begin{proof}[Proof of Theorem~\ref{thm:mainEng}]
By the local normal form for Engel structures it is enough to work in $(\R^4,\Dst)$. Let $T_0$ be a del Pino--Vogel transverse torus of the form $C\times K_0$ where $C$ is a transverse curve in $(\R^4,\Dst)$ and $K_0$ a transverse knot with self-linking number $\slf(K_0)$ in the $(D^3_R,\ker dy-wdx)$-slice of a standard neighborhood $(\nu C,\Dst)$ of $C$.
We denote by $T_n$ the $n$-fold stabilization of $T_0$, i.e.\ the transverse torus given by $C\times K_n$ where $K_n$ denotes the $n$-fold stabilization of $K_0$ with self-linking number $\slf(K_n)=\slf(K_0)-n$.

Note that since $K_n$ is smoothly isotopic to $K_0$ it follows that $T_n$ is smoothly isotopic to $T_0$, for all $n\in\N_0$. Next, we want to show that the $T_n$ are pairwise non-isotopic as transverse tori by comparing their self-linking classes.
For that we use Lemma~\ref{lem:computationofsl} to compute the self-linking classes for some choice of $\alpha$ and $\beta$ as 
\begin{equation*}
\slf(T_n)=\slf(K_n) \cdot	\alpha=\big(\slf(K_0)-n\big)\cdot\alpha.
\end{equation*}
But since $\alpha$ is a primitive element in a free abelian group its divisibility $\slf(K_n)$ is independent of the choice of $\alpha$ and $\beta$ and thus distinguishes the transverse tori $T_n$ pairwise.
\end{proof}

\section{Formal transverse tori and their self-linking classes}\label{sec:formal}
We recall that a formal Engel structure on $M$ (in the oriented setting) consists of three oriented and cooriented subbundles $(\W^1,\D^2,\E^3)$ of $TM$ such that $\W\subset\D\subset\E\subset TM$.
A \textbf{formal transverse torus} in a formal Engel manifold $(M;\W,\D,\E)$ is an embedding $f\colon T^2\rightarrow M$ covered by a family of injective bundle maps $(F_s)_{s\in [0,1]}\colon TT^2\rightarrow TM$ such that $F_0=Tf$ and $F_1(TT^2)$ is transverse to $\D$. We have obvious inclusions of the space of Engel structures and the space of transverse tori into their formal counterparts. Next, we generalize the definition of the self-linking class to formal transverse tori.

Let $T=(f,(F_s)_{s\in[0,1]})$ be a formal transverse torus in a formal Engel manifold $(M;\W,\D,\E)$ and set $\W_1:=\W$. Since $\W_1$ is transverse to $F_1(TT^2)$ we can apply the homotopy lifting property to obtain unique (up to homotopy) line fields $\W_s$ that are transverse to $F_s(TT^2)$, for all $s\in[0,1]$. Then $\W_0$ is transverse to $T=f(T^2)$ and we denote by $T'$ the $2$-torus which is obtained by pushing $T$ into $\W_0$-direction. We define the \textbf{self-linking class} $\slf(T)$ of $T$ as the homology class
$$\slf(T):=\lk(T',T)\in H_2(M\setminus T).$$
Since $\W_0$ is unique up to homotopy this is well-defined and generalizes the previous definition.

We will now demonstrate that the self-linking class does not change under homotopy through formal transverse tori and also does not change through formal homotopy of the ambient Engel manifold. 


\begin{lem}
Let $(\W_t,\D_t,\E_t)_{t\in[0,1]}$ be a family of formal Engel structures on a closed $4$-manifold $M$ and let $(T_t)_{t\in[0,1]}$ be a family of formal transverse $2$-tori such that $T_t$ is formally transverse to $\D_t$. Then there is a smooth ambient isotopy that identifies $H_2(M\setminus T_0)$ and $H_2(M\setminus T_1)$. Under this identification the self-linking class $\slf(T_0)$ of $T_0$ equals the self-linking class $\slf(T_1)$ of $T_1$.
\end{lem}

\begin{proof}
Since $T_t$ is a smooth family of embedded $2$-tori in $M$, the isotopy extension theorem guarantees the existence of a smooth ambient isotopy of the $T_t$, i.e. a family of diffeomorphisms $(G_t)_{t\in[0,1]}\colon M \rightarrow M$ such that $G_0=\operatorname{Id}_M$ and $G_t(T_0)=T_t$. By restricting this ambient isotopy to the complements we obtain a family of diffeomorphisms, for simplicity again denoted by $(G_t)_{t\in[0,1]}\colon M\setminus T_0\rightarrow M\setminus T_t$ such that $G_0=\operatorname{Id}_{M\setminus T_0}$. We use the induced isomorphism on homology to identify $H_2(M\setminus T_t)$ with $H_2(M\setminus T_0)$, for any $t\in[0,1]$.  

Now let $T_t=(f_t,(F_s^t)_{s\in[0,1]})$, where $f_t\colon T^2\rightarrow M$ is an embedding and $(F_s^t)_{s\in[0,1]}\colon TT^2\rightarrow TM$ denotes a family of injective bundle maps with $F_0^t=Tf_t$ and $F_1^t(TT^2)$ transverse to $\D_t$. As in the definition of the self-linking number for formal transverse tori we obtain a family of line fields $\W_t^s$ such that $\W_t^1=\W_t$ and $\W_t^s$ is transverse to $F_s^t(TT^2)$. Applying again the homotopy lifting property we see that $\W_t^0$ is homotopic to $TG_t(\W_0^0)$. But this implies that the homology class $[T_0']$ gets mapped under $(G_t)_*$ to $[T_t']$ and thus the self-linking classes of $T_0$ and $T_1$ get identified.
\end{proof}


\section{Legendrian tori in even contact manifolds}
\label{section:EvenLegendrian} 

In this section we will briefly initiate the study of tangential embeddings of maximal dimension in even contact structures. 

\begin{defi}
An embedded $2$-torus $L$ in an even contact $4$-manifold $(M,\E)$ is called \textbf{Legendrian} if it is everywhere tangent to $\E$, i.e. $TL\subset\E$.
\end{defi}

It is possible to construct examples coming from Legendrian knots in contact $3$-manifolds as follows.

\begin{ex}
Let $K$ be a Legendrian knot in a contact $3$-manifold $(N,\xi)$. Then 
$$L=S^1\times K\subset (S^1\times N,\E:=\partial_\theta\oplus\xi)$$
is a Legendrian torus, where $\theta$ denotes an angular coordinate on $S^1$.
\end{ex}

Observe that in these examples the characteristic line field $\W$ is tangent to $L$. This is always the case: If $p$ is a point in $L$ where $L$ is transverse to $\W$ then $L$ is transverse in a whole open neighborhood of $p$. But this neighborhood projects along $\W$ to a surface $F$ in a contact $3$-manifold $(N,\xi)$ such that $F$ is tangent to $\xi$ implying that such a point $p$ cannot exist. In particular, we do not have any Legendrian tori in the local normal form $(\R^4,\Est)$ of an even contact structure. More generally, it is known that the $\W$-orbits of a generic Engel manifolds are isolated~\cite[Theorem~27]{PP19}, cf.~\cite{Mo99}, and thus a generic Engel manifold contains no Legendrian $2$-torus.

By a similar argument as for transverse surfaces, we argue that a general closed Legendrian surface has to be a $2$-torus if it is orientable and a Klein bottle if it is non-orientable.
We now define a homological invariant for Legendrian tori from the linking class as follows.

\begin{defi}
Let $L$ be a Legendrian torus with trivial normal bundle in $(M,\E)$. We denote by $L'$ the embedded $2$-torus obtained by pushing $L$ into a direction transverse to $\E$, i.e. in the $TM/\E$-direction.
The \textbf{Thurston--Bennequin class} $\tb(L)$ is defined as
$$\tb(L)=\lk(L',L)\in H_2(M\setminus L).$$
\end{defi}


\begin{proof}[Proof of Theorem~\ref{thm:mainEven}]
	Let $K_0$ be a nullhomologous Legendrian knot with Thurston--Bennequin invariant $\tb(K_0)$ in a contact $3$-manifold $(N,\xi)$ and $L_0$ the Legendrian torus $S^1\times K_0$ in $(M=S^1\times N,\E=\partial_\theta\oplus\xi)$. We write $L_n$ for the Legendrian torus $S^1\times K_n$, where $K_n$ is some $n$-fold stabilization of $K_0$ with Thurston--Bennequin invariant $\tb(K_n)=\tb(K_0)-n$. We note that the $L_n$ are all smoothly isotopic since the $K_n$ are smoothly isotopic.
	
By Lemma~\ref{lem:hom2} we know that $\tb(L_n)$ is a multiple of $\mu_{K_0}$. Next we compute this multiple. The push-off $L'_n$ is given by $S^1\times K'_n$ where $K'_n$ is the push-off of $K_n$ in Reeb-direction of $\xi$. A computation in homology yields (see Figure~\ref{fig:Legendretori})
\begin{equation*}
\tb(L_n)=[L'_n]=[S^1\times K'_n]=\tb(K_n)\cdot\mu_{K_0}=\big(\tb(K_0)-n\big)\cdot\mu_{K_0}\in H_2(M\setminus L_0)
\end{equation*}
and thus the Legendrian tori $L_n$ are all pairwise non-isotopic.
\end{proof}

\begin{figure}[htbp] 
	\centering
	\def\svgwidth{0,99\columnwidth}
	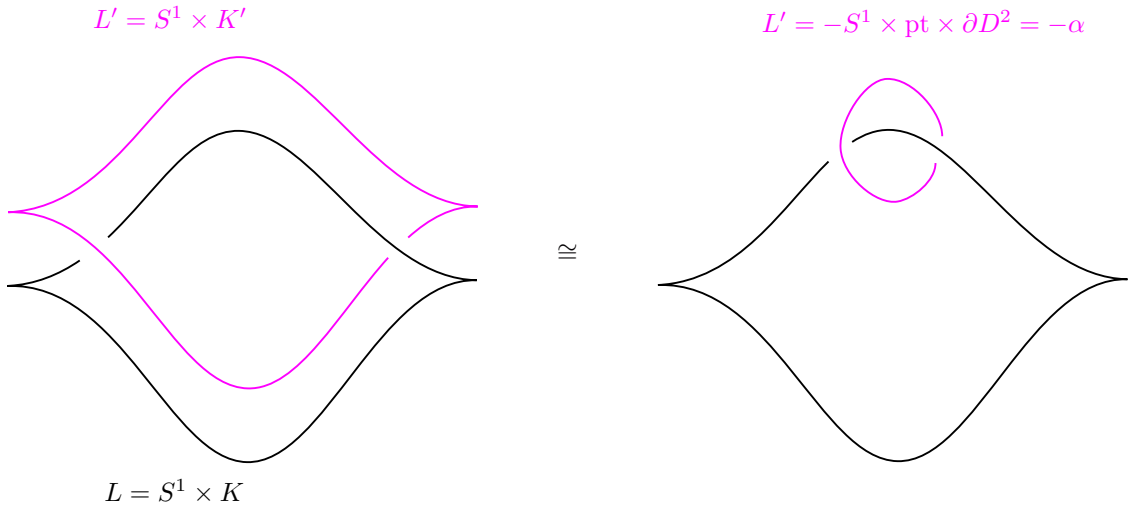
	\caption{The push-off $L'$ is isotopic to $\tb(K)\alpha$.}
	\label{fig:Legendretori}
\end{figure}



\end{document}

%% file: transversetori.pdf_tex
\begingroup%
  \makeatletter%
  \providecommand\color[2][]{%
    \errmessage{(Inkscape) Color is used for the text in Inkscape, but the package 'color.sty' is not loaded}%
    \renewcommand\color[2][]{}%
  }%
  \providecommand\transparent[1]{%
    \errmessage{(Inkscape) Transparency is used (non-zero) for the text in Inkscape, but the package 'transparent.sty' is not loaded}%
    \renewcommand\transparent[1]{}%
  }%
  \providecommand\rotatebox[2]{#2}%
  \ifx\svgwidth\undefined%
    \setlength{\unitlength}{494.17519748bp}%
    \ifx\svgscale\undefined%
      \relax%
    \else%
      \setlength{\unitlength}{\unitlength * \real{\svgscale}}%
    \fi%
  \else%
    \setlength{\unitlength}{\svgwidth}%
  \fi%
  \global\let\svgwidth\undefined%
  \global\let\svgscale\undefined%
  \makeatother%
  \begin{picture}(1,0.59536024)%
    \put(0,0){\includegraphics[width=\unitlength,page=1]{transversetori.pdf}}%
    \put(0.07229268,0.59386039){\color[rgb]{1,0,0}\makebox(0,0)[lt]{\begin{minipage}{0.49058837\unitlength}\raggedright $\alpha=S^1\times\text{pt}\times\partial D^2$\end{minipage}}}%
    \put(0.5977548,0.60019144){\color[rgb]{0.00784314,0,0.99215686}\makebox(0,0)[lt]{\begin{minipage}{0.39247071\unitlength}\raggedright $\beta=\text{pt}\times K\times \partial D^2$\end{minipage}}}%
    \put(0.27902636,0.42836225){\color[rgb]{0.00392157,0,0}\makebox(0,0)[lt]{\begin{minipage}{0.18856989\unitlength}\raggedright $T=S^1\times K$\end{minipage}}}%
    \put(0.07962705,0.02704509){\color[rgb]{0.99215686,0,0.99607843}\makebox(0,0)[lt]{\begin{minipage}{0.20850005\unitlength}\raggedright $T'=S^1\times K'$\end{minipage}}}%
    \put(0.81656297,0.02863981){\color[rgb]{0.99215686,0,0.99607843}\makebox(0,0)[lt]{\begin{minipage}{0.15024269\unitlength}\raggedright $T'=-\alpha$\end{minipage}}}%
    \put(0.49489139,0.16493064){\color[rgb]{0.01176471,0,0.01176471}\makebox(0,0)[lt]{\begin{minipage}{0.09658457\unitlength}\raggedright $\cong$\end{minipage}}}%
  \end{picture}%
\endgroup%

%% file: Legendretori.pdf_tex
\begingroup%
  \makeatletter%
  \providecommand\color[2][]{%
    \errmessage{(Inkscape) Color is used for the text in Inkscape, but the package 'color.sty' is not loaded}%
    \renewcommand\color[2][]{}%
  }%
  \providecommand\transparent[1]{%
    \errmessage{(Inkscape) Transparency is used (non-zero) for the text in Inkscape, but the package 'transparent.sty' is not loaded}%
    \renewcommand\transparent[1]{}%
  }%
  \providecommand\rotatebox[2]{#2}%
  \ifx\svgwidth\undefined%
    \setlength{\unitlength}{486.05372524bp}%
    \ifx\svgscale\undefined%
      \relax%
    \else%
      \setlength{\unitlength}{\unitlength * \real{\svgscale}}%
    \fi%
  \else%
    \setlength{\unitlength}{\svgwidth}%
  \fi%
  \global\let\svgwidth\undefined%
  \global\let\svgscale\undefined%
  \makeatother%
  \begin{picture}(1,0.40684381)%
    \put(0,0){\includegraphics[width=\unitlength,page=1]{Legendretori.pdf}}%
    \put(0.44722327,0.21824365){\color[rgb]{0.01176471,0,0.01176471}\makebox(0,0)[lt]{\begin{minipage}{0.09819841\unitlength}\raggedright $\cong$\end{minipage}}}%
    \put(0.06979342,0.41175574){\color[rgb]{0.99215686,0,0.99607843}\makebox(0,0)[lt]{\begin{minipage}{0.21198387\unitlength}\raggedright $L'=S^1\times K'$\end{minipage}}}%
    \put(0.0791951,0.02675558){\color[rgb]{0.00392157,0,0}\makebox(0,0)[lt]{\begin{minipage}{0.1917207\unitlength}\raggedright $L=S^1\times K$\end{minipage}}}%
    \put(0.61371103,0.40840567){\color[rgb]{1,0,1}\makebox(0,0)[lt]{\begin{minipage}{0.39279358\unitlength}\raggedright $L'=-S^1\times\text{pt}\times\partial D^2=-\alpha$\end{minipage}}}%
  \end{picture}%
\endgroup%